%% file: main.tex
\documentclass[a4paper,11pt,reqno]{amsart}
\pdfoutput=1
\input{packages.tex}

\title{The discrete Gaussian free field on a compact manifold}

\author[A.~Cipriani]{Alessandra Cipriani}
\author[B.~van Ginkel]{Bart van Ginkel}
\address{TU Delft (DIAM), Building 28, van Mourik Broekmanweg 6, 2628 XE, Delft, The Netherlands}
\email{A.Cipriani@tudelft.nl, G.J.vanGinkel@tudelft.nl}
\thanks{The authors thank Mario Ayala Valenzuela, Nathana\"el Berestycki, Rajat Subhra Hazra and Frank Redig for helpful discussions. The support of the grants 613.009.102 and 613.009.112 of the Netherlands Organisation for Scientific Research (NWO) is gratefully acknowledged. The authors also would like to thank an anonymous referee for insightful comments and remarks on a previous version of the paper.
}
\date{\today}

\begin{document}
\maketitle

%CONTENT---------------------------------------------------

\begin{abstract}
    In this article we define the discrete Gaussian free field (DGFF) on a compact manifold. Since there is no canonical grid approximation of a manifold, we construct a random graph that suitably replaces the square lattice $\Z^d$ in Euclidean space, and prove that the scaling limit of the DGFF is given by the manifold continuum Gaussian free field (GFF). Furthermore using Voronoi tessellations we can interpret the DGFF as element of a Sobolev space and show convergence to the GFF in law with respect to the strong Sobolev topology.
\end{abstract}

%\let\clearpage\relax
%{\color{red} CHANGE THE FOLLOWING:\\
%- \sout{SCALING OF $\phi_N$ AND DEFINITION OF $(\phi_N,f)$}\\
%- \sout{CHANGE REFERENCE TO van GINKEL-REDIG}\\
%- \sout{ADD DISCRETE MARKOV PROPERTY}\\
%- \sout{RIVERA REFERENCE}\\
%- \sout{DEFINITION GREEN'S FUNCTION 2.1}}
\input{1intro.tex}
\input{2definitions.tex}

\input{3main.tex}

\input{4voronoi.tex}

%----------------------------------------------------------

%\phantomsection
\bibliographystyle{abbrvnat}
\bibliography{refs}

\end{document}

%% file: packages.tex
\usepackage{amsmath,amsthm,mathrsfs}
\usepackage{bbm}
\usepackage{amsfonts}
\usepackage{amssymb}
\usepackage{xcolor}
\usepackage[titletoc]{appendix}
\usepackage{enumitem}

\usepackage[a4paper,vmargin={3.5cm,3.5cm},hmargin={2.5cm,2.5cm},bottom=30mm]{geometry} %This sets better margins

\usepackage[comma, sort&compress, authoryear]{natbib} %For bibliography
\usepackage[osf,sc]{mathpazo} %fonts
\usepackage{graphicx}
\usepackage[utf8]{inputenc}

\usepackage[final]{changes}
\definechangesauthor[name={ale}, color=orange]{ale}
\definechangesauthor[name={bart}, color=purple]{bart}

\newcommand{\R}{\mathbb{R}}
\newcommand{\N}{\mathbb{N}}

\newcommand{\Z}{\mathbb{Z}}
\newcommand{\p}{\mathbb{P}}
\newcommand{\E}{\mathbb{E}}
\newcommand{\var}{\mathrm{var}}

\newcommand{\e}{\mathrm{e}}
\newcommand{\dd}{\mathrm{d}}

\newcommand{\1}{\mathbbm{1}}
\newcommand{\im}{\mathrm{i}}
\newcommand{\V}{\mathcal{V}}
\newcommand{\U}{\mathcal{U}}

\renewcommand{\phi}{\varphi}
\newcommand{\tphi}{\widetilde{\varphi}}
\renewcommand{\bar}{\overline}

\newcommand{\Ric}{\text{Ric}}

\usepackage{hyperref}% This is for hyperlinks
\hypersetup{ colorlinks=true, linkcolor=blue, citecolor=blue,
  filecolor=blue, urlcolor=blue, breaklinks=true}

  \numberwithin{equation}{section} %This is for enumerating equations

\makeatletter
\def\paragraph{\@startsection{paragraph}{4}%
  \z@\z@{-\fontdimen2\font}%
  {\normalfont\itshape}}
  \makeatother

  \makeatletter
\renewcommand\subsection{\@startsection{subsection}{2}%
  \z@{-.5\linespacing\@plus-.7\linespacing}{.5\linespacing}%
  {\normalfont\scshape}}

 \makeatletter
\def\subsubsection{\@startsection{subsubsection}{3}%
  \z@\z@{-\fontdimen2\font}%
  {\normalfont\scshape}}
  \makeatother
\theoremstyle{definition}
\newtheorem{defi}{Definition}[section]
\newtheorem{rmk}[defi]{Remark}
\newtheorem{lemma}[defi]{Lemma}

\theoremstyle{plain}
\newtheorem{thm}[defi]{Theorem}
\newtheorem{prop}[defi]{Proposition}
\newtheorem{cor}[defi]{Corollary}

%% file: 1intro.tex
\section{Introduction and main results}
The discrete Gaussian free field has received a lot of attention over the last years thanks to its connections with several areas of mathematics. An on-the-fly definition of it can be given by means of a multivariate centered Gaussian variable on a finite graph, whose covariance matrix is the inverse of the graph laplacian. The DGFF is considered the discrete version of a random distribution, the Gaussian free field, and the interplay between the two has been highlighted in the mathematics literature starting with the work of \cite{Sheffield:2007}. As far as the authors know, the DGFF has been considered mainly on lattices due to the reason that, outside of the Euclidean setting, it is difficult to choose a canonical grid that approximates space (see the question on~\cite{Stackex}). If one wants to construct the DGFF on a Riemannian manifold for example, one possible strategy to define it is to begin directly with the GFF on the manifold, then construct a triangulation of the space and project the GFF on test functions that are affine on triangles. This procedure is originally contained in~\cite{Schramm/Sheffield:2013}. The drawback of this construction is that it does not link the DGFF to a metrized graph, in particular \added[id=ale]{it does not give information on the edge weights that the underlying graph should have. We, on the other hand, start by setting the edge weights and from there constructing the DGFF.} {Another approximation of the GFF is obtained via a truncation of its Wiener series representation in terms of eigenfunctions of the Laplace--Beltrami operator, as done in~\cite{Riv17}. This approach is analytical and does not yield a DGFF, which is the object we want to use to discretize the GFF.}
\paragraph{Contributions of this article} Indeed, the goal of our work is to approximate the GFF on a manifold by an appropriately defined DGFF. The main difficulty here is to create the right setting in which to make the necessary constructions. In particular this means finding a graph on which we can define the DGFF. However, grids on manifolds are in general far from regular (translation invariance and scaling properties are usually lacking). Studying the analogous proofs in $\R^d$ shows that these are key ingredients used in the scaling limit. Another object which plays a crucial role in the Euclidean case is the Green's function, for which one needs pointwise convergence to the continuum Green's function and upper bounds. These are however not available for our weighted graphs. Therefore we aimed at, and succeeded in, finding different conditions under which Green's functions converge in a weaker way, but still strong enough to ensure the scaling limit. These assumptions (as listed in Theorem~\ref{thm:1} and as elaborated upon shortly) are natural, in the sense that a grid uniformly sampled  from the manifold exhibits these properties with probability one (see Theorem~\ref{thm:2}). In contrast to the $\R^d$ case, showing this relies on a result on the spectral convergence of graph laplacians, which is found in the literature of spectral clustering.

Our construction starts by considering a sequence of weighted graphs for which the random walk semigroups converge in the sense described in Theorem~\ref{thm:1}. Another quite natural assumption we make is that the grid approximates the manifold in the sense of measures, that is, the empirical measure on the grid points converges to the uniform measure on the manifold. Given this, we have to add one final ingredient to the picture: a uniform bound from below on the spectral gap of the discrete laplacians. The reason behind this condition is that one wishes to stay in the region of the spectrum away from zero, where the graph laplacian is invertible. 

We can now begin by giving the mathematical exposition of our results. Throughout we will be working with a connected and compact Riemannian manifold $M$ of dimension $d\ge 1$ with normalized volume measure $\overline V$.
We will use the space of smooth and zero-mean test functions, that is to say the set
\[
W:=\left\{f\in C^\infty(M):\,\int_M f\dd \overline V=0\right\}.
\]

For a graph $\V$ with positive symmetric edge weights $c_{vw}$ we define the graph laplacian acting on functions $f:\V\rightarrow \R$ as
\begin{equation}\label{eq:graph_Lapl}
    Lf(v):=-\sum_{w\in \V} c_{vw} (f(w)-f(v)), \hspace{0.5cm} v\in \V.
\end{equation}
The laplacian $L$ generates a simple random walk on $\V$ with associated semigroup $(S_t^\V)_{t\ge 0}$.
We define the zero-average discrete Gaussian free field $\phi_{\V}$ on $\V$ as the Gaussian field indexed by $\V$ whose covariance function is the inverse of $L$ (for proper definitions see Subsection~\ref{sec:DGFF}).

The first Theorem we present is concerned with the convergence of the zero-average DGFF to its continuum counterpart: the Gaussian free field on $M$, that is, the generalized Gaussian field $\phi$ with mean zero and covariance matrix $G$, the Green's function of the Laplace--Beltrami operator on $M$ (these notions will be specified in Section~\ref{sec:overview}). While the first two conditions in the Theorem specify how to choose a suitable graph laplacian approximating the Laplace--Beltrami operator, the third one regards the dispersion of the grid points.
\begin{thm}\label{thm:1}
Let a sequence of graphs $(V_N)_{N\in\N}$ be given such that $V_N=(p_i^N)_{i=1}^N$ and set the graph laplacians $L=L_N$ as in \eqref{eq:graph_Lapl}. Define $\phi_N$ to be the zero-average DGFF on $V_N$. Assume that the following conditions hold.
\begin{enumerate}[ref=(\arabic*)]
    \item\label{as:gap} Denoting by $\lambda_2^N$ the spectral gap of $L_N$, we require $\inf_N \lambda_2^N>0. $
    \item\label{as:semigroup} For any $f:\,M\to \R$, set $$f_N:=f|_{V_N}-\frac{1}{N}\sum_{i=1}^Nf|_{V_N}(p_i^N).$$
    and assume that for all $f\in W$ and $t\ge 0$
    \begin{equation}\label{eq:semigroup_conv}
    \lim_{N\to\infty}\frac{1}{N}(f_N,S_t^N f_N) = (f,\,S_t f),
\end{equation}
where $(S_t)_{t\ge 0}$ is the heat semigroup of the Laplace-Beltrami operator.
\item\label{as:empirical} The following weak limit of measures holds:
\[
\lim_{N\to\infty}\frac{1}{N}\sum_{i=1}^N\delta_{p_i^N}=\overline V,
\]
where $\overline V$ is the uniform measure on $M$.
\end{enumerate}
Then $\sqrt N\phi_N$ converges to $\phi$ in law in the space $W'$ equipped with the weak* topology.
\end{thm}
We will show (see Remark~\ref{rmk:torus}) that canonical grids in flat space satisfy the above mentioned assumptions, for example the equally spaced grid on the $d$-dimensional flat torus $\mathbb T^d$. 
\begin{rmk}
It will follow from our proofs that we do not necessarily have to work with the Laplace--Beltrami operator. In general, two properties are essential: first of all the operator needs to be symmetric and positive semi-definite. This ensures that we can use its (possibly generalized) inverse as covariance of a Gaussian field, as we are going to do in Section~\ref{sec:overview}. Further the operator must generate a suitably regular semigroup for our approach to work. Then if we have a sequence of discrete approximations of this operator in the sense of Theorem~\ref{thm:1} with the analogous properties, we get convergence of the corresponding Gaussian fields.
\end{rmk}

The second Theorem exhibits an example of a graph satisfying Assumptions~\ref{as:gap}-\ref{as:empirical}. As it often happens in statistics and manifold learning (\cite{singer2006graph, Belkin/Niyogi:2006,H/A/VL:2005,Gine/Koltchinskii:2006} are only a few of the numerous works on the topic), the points $(p_i^N)$ of the grid are obtained as uniform observations of the manifold, and edges between them are weighted by a semi-positive kernel with
bandwidth $t$ applied to the distance between those grid points. As the number of observations grows and the bandwidth goes to zero, one should be able to capture the convergence of the graph laplacian to the continuum one, and in turn the scaling limit of the random field. Concretely, we sample points uniformly from $\overline V$ and we define the vertex set of the $N^\text{th}$ grid to be the first $N$ points. We connect any two vertices with an edge and choose our kernel to be the heat kernel $p_t(\cdot,\,\cdot)$ on $M$ divided by $t$ (the more precise definitions are in Subsection~\ref{subsec:manifold_as}). Given the sequence of grids we set a bandwidth $t$ that satisfies
\begin{equation}\label{eq:Wass_bound}
W_1(\mu^N,\,\overline V)=o\left(t^{\frac{d}{2}+2}\right),
\end{equation}
where $W_1$ denotes the Kantorovich or $1$-Wasserstein metric and $\mu^N$ is the empirical measure on $(p_i^N)_{i=1}^N.$
%We remark in particular that this choice of weights makes $V_N$ a (weighted) complete graph. The fact that all possible edges are present in the graph is due to symmetrisation: as the square lattice is translation invariant, one hopes to find an analogous grid without any special vertices (for example, with a higher degree than others). 
Finally, we modify the bandwidth so that it goes to $0$ slowly enough to get convergence of the spectral gaps of the graph laplacians to the continuum one (see Subsection~\ref{subsubsec:as2holds} for the details). We formulate the result in the following Theorem.
\begin{thm}\label{thm:2}
Let $V_N:=(p_i^N)_{i=1}^N$ be a sequence of i.i.d. points sampled from the normalized volume measure on $M$. Let $p_t(\cdot,\,\cdot)$ be the heat kernel on $M$. Choose $t_N$ such that~\eqref{eq:Wass_bound} holds and the spectral gaps converge to the continuum one. Define the weights in~\eqref{eq:graph_Lapl} as
\[
c_{vw}:=\frac{p_{t_N}(v,\,w)}{Nt_N},\quad v,\,w\in V_N.
\] 
Then Assumptions~\ref{as:gap}-\ref{as:empirical} are satisfied almost surely in the law of the sampled grid points.
\end{thm}
Finally we extend the result to convergence in a stronger sense, namely in the Sobolev space {$H^{-s}(M)$} for some $s>0$. To do this we lift $\phi_N$ {using Voronoi cells with centers $(p_i)_{i=1}^N$ to a random distribution in $H^{-s}(M)$ by specifying the action}
\[
\langle \widetilde \phi_N,\,f\rangle:=\frac{1}{N}\sum_{i=1}^N \phi_N(p_i)\frac{1}{v_i}\int_{C_i}f(p)\overline{V}(\dd p)
\]
{with $v_i$ the volume of the cell $C_i$.}
{We also extend the definition of the CGFF $\phi$ to let it act on $H^k$ functions.}  Then we get the following theorem.
\begin{thm}\label{thm:vor}
Assume the conditions of Theorem~\ref{thm:1}. Then $\sqrt N \widetilde\phi_N $ converges to $ \phi$ in law in the strong topology of $H^{-s}$ for $s>d-1/2$.
\end{thm}

\paragraph{Structure of the paper}In Section~\ref{sec:overview} we will give the precise definitions of the Gaussian fields we consider, as well as the necessary background on the geometry of the manifold and further insight on Assumptions~\ref{as:gap}-\ref{as:empirical}. Section~\ref{sec:proof} is devoted to showing the first two main Theorems, respectively in Subsections~\ref{subsec:thm1} and ~\ref{subsec:2}. The result in $H^{-s}(M)$ is stated and proved in Section~\ref{subsec:vor}.
\paragraph{Notation} In the following we will use $C,\,c,\,c',\,\ldots$ as absolute constants whose value may change from line to line even within the same equation. The norms with subscript $N$ are those on the graphs $V_N$. We will also use square brackets to denote dual pairings and round brackets for inner products.

%% file: 2definitions.tex
\section{Preliminaries: definitions and assumptions}\label{sec:overview}

\subsection{The manifold}\label{subsec:manifold_as}
We assume $M$ to be a compact, connected and $d$-dimensional Riemannian manifold (for all of the following definitions see for instance~\cite{grigoryan2009heat}). The Riemannian structure induces the metric $d(\cdot,\cdot)$. We denote the volume measure on $M$ by $V$ and the uniform measure by $\bar V:={V}/{V(M)}$ (note that $M$ is compact, so $V(M)<\infty$). On $M$ we can define the heat semigroup\footnote{Note that one can construct the heat semigroup on either $C(M)$ or $L^2(M)$. We will need both representations in what follows. However since we will evaluate the semigroups on the set of smooth functions, where they agree, we do not need to specify which one we are using.} $(S_t,\,t\geq 0)$ generated by the Laplace--Beltrami operator $\Delta_M$ and the corresponding heat kernel $p_t(p,\,q)$ such that
\begin{equation*}
    S_tf(p) = \int_M p_t(p,\,q)f(q)\overline V (\dd q),\quad f\in L^2.
\end{equation*}
Recall from the introduction that $W\subset C^\infty(M)$ consists of the zero-average smooth functions on $M$. It is equipped with the topology that is generated by the seminorms
\begin{equation*}
    \sup_K |\partial^\alpha u|,
\end{equation*}
where $K$ ranges over the compact sets that are contained in charts and $\partial^\alpha$ ranges over partial derivatives in charts containing $K$.
For $f,\,g\in L^2(M)$ we denote
\begin{equation*}
    (f,g):=\int_M f(p)g(p) \dd \overline V.
\end{equation*} 
We recall some basic facts on the Green's function of $-\Delta_M$ (for more details we refer the reader to~~\citet[Chapter~4]{aubin},~\citet{donaldson},~\citet[Chapter 13]{grigoryan2009heat}). One knows that on a compact manifold the spectrum of $-\Delta_M$ is discrete, and is given by $0=\lambda_1<\lambda_2\le \lambda_3\ldots$ The Green kernel on $M$ is given by the following sum in $L^2(M)$:
\begin{equation}\label{eq:decom_G2}
G:=\sum_{j\ge 2}\frac{1}{\lambda_j}P_j
\end{equation}
with $P_j$ the projection on the $j$-th eigenspace of $-\Delta_M$. %Equivalently, one can say that
%\begin{equation*}
%    G f := \begin{cases} (-\Delta_M)^{-1}f & f\perp \1 \\
 %   0 &  \text{otherwise} 
 %   \end{cases}.
%\end{equation*}
We also recall that on a compact Riemannian manifold without boundary
%if $(\phi_j)_{j\ge 1}$ is an orthonormal basis of $L^2(M,\,\dd \overline V)$ formed by eigenfunctions of the negative Laplace--Beltrami, the function
%\begin{equation}\label{eq:decom_G2}
%G(p,\,q)=\sum_{j\ge 2}\frac{1}{\lambda_j}\phi_j(p)\phi_j(q),\quad p,\,q\in M
%\end{equation}
%is such that
%\[
%f(p):=\int_M G(p,\,q)\rho(q)\overline V(\dd q)
%\]
$f=G\rho$ solves $-\Delta_M f=\rho$ for the input datum $\rho\in W$ and the solution is normalized to have integral zero. Moreover in that case $f\in W$.
\subsection{The zero-average discrete Gaussian free field}\label{sec:DGFF} We will now recall some definitions concerning the discrete Gaussian free field. The idea behind the construction follows the use of fundamental matrices to define Gaussian processes \cite[Section 14.6.2]{Aldous/Fill} and has been applied for example in studying the zero-average DGFF on the torus by \cite{Abacherli:2017} .

Let $\V$ be a finite graph. For $v,\,w\in \V$, let $c_{vw}=c_{wv}\geq 0$ be the conductance between $v$ and $w$. Assume that $\V$ is connected in the sense that for any $v,w\in \V$ there is a path from $v$ to $w$ such that each edge that is traversed has strictly positive conductance. We define the graph laplacian acting on functions $f:\V\rightarrow\R$ by
\begin{equation*}
    Lf(v)=-\sum_{w\in \V} c_{vw} (f(w)-f(v)), \hspace{0.5cm} v\in \V.
\end{equation*}
Since the graph is symmetric all the eigenvalues are non-negative and the corresponding eigenspaces are orthogonal. Moreover, we can conclude from the connectedness that there is exactly one eigenvalue $0$ (see for instance~\citet[Chapter 1]{chung1997spectral}) with eigenfunction the constant function $\1$. Because of this, the following definition makes sense.
\begin{defi}[Discrete Green's function]
{We define the Green's operator as the linear operator on functions $f:\V\rightarrow \R$ uniquely defined by the following action on two linear subspaces}
\begin{equation*}
    G^\V f := \begin{cases} L^{-1}f & f\perp \1 \\ 0 &  f=c\1 \end{cases}.
\end{equation*}
\end{defi}
Here $\1$ is the function constantly equal to one. There is also an explicit characterization of $G^\V$, which we are going to use in the following. Assume that $\V$ has $n$ points. Denote the eigenvalues of $L$ by $0=\lambda_{1}^n< \lambda_{2}^n\le\,\ldots\le\lambda_{n}^n$, possibly with multiplicities. Since $\1$ is exactly the eigenspace corresponding to $\lambda_{1}^n$ we can write
\begin{equation}\label{spect_rep_G_N}
    G^\V=\sum_{j=2}^n \frac{1}{\lambda_{j}^n} P_{j}^n.
\end{equation}
Here $P_{j}^n$ is the projection on the eigenspace corresponding to the $j$-th eigenvalue of $L$.

Now that we have introduced the Green's function, we can make the following definition.
\begin{defi}[DGFF as a multivariate Gaussian]\label{defi:multivariate}
The zero-average Gaussian free field $\phi_\V$ on $\V$ is the Gaussian vector indexed by $\V$ with mean $0$ and covariance matrix $G^\V$. 
\end{defi} 
Note that $G^\V$ is symmetric and positive definite on $\{f\perp\1\}$ (since $L$ is) and $0$ on the rest. Therefore $\phi_\V$ lives in an $(n-1)$-dimensional space and is degenerate in the direction of the constant vectors. Indeed, as the name indicates, $\phi_\V$ has average $0$ almost surely. One can see this since
\begin{equation*}
    \var\left(\sum_{v\in \V}\phi_\V(v)\right)=\1^TG\1 = 0
\end{equation*}
so
\[
\frac{1}{|\V|}\sum_{v\in \V} \phi_\V(v) = 0 \quad\text{ a.s.}
\]

One of the most important properties of the DGFF is the Markov property, i.e. that the DGFF restricted to a subset of the underlying graph only depends on the rest of the graph through the boundary of that subset~\cite[Proposition~2.3]{ASS:2012}. In a zero-average DGFF this is no longer true, since the total average should be zero. Moreover, the restriction of a zero-average DGFF to a subset is not even a zero-average DGFF. However, we can still study the restriction of the zero-average DGFF to a subset when we subtract the harmonic interpolation of its values on the boundary. This turns out to be a DGFF, as it is shown in~\citet[Lemma 1.7]{Abacherli:2017} for the zero-average DGFF on the torus. The same proof works in our case, given a few generalizations of the definitions that are involved. We will now formulate the statement. To this end, let $X=(X_t,t\geq 0)$ denote the random walk on $\V$ generated by $-L$, denote by $\E_v$ and $\p_v$ the expectation and law of $X$ started from $v\in \V$, respectively, and set $T_U=\inf\{t:X_t\notin U\}$ .
\begin{lemma}\label{lem:Angelo}
Let $\U\subset\V$ be a proper subset and for $v\in\V$ define
\begin{equation*}
    \phi^\U(v):=\phi^\V(v)-\E_v[\phi^\V(X_{T_\U})].
\end{equation*}
Then $\phi^\U$ is a centered Gaussian field with covariance matrix
\begin{equation*}
    G^\U(v,w)=\E_v\left[\int_0^{T_\U}\1_{\{X_t=w\}}\dd t\right].
\end{equation*}
\end{lemma}
\begin{proof}The proof of this Lemma is essentially the same of \citet{Abacherli:2017}, with two main remarks that we want to stress now. Firstly let us note that we have to use the continuous-time random walk (as opposed to the situation in~\cite{Abacherli:2017}), since the rates of the exponential weighting times do not have to be equal. Secondly, to be able to mimic the proof given on the $d$-dimensional flat torus we need to prove that our Green's function is the same as \citet{Abacherli:2017}, i.e. that
\begin{equation}\label{eq:key_Angelo}
    G^\V(v,w)=\int_0^\infty \left(\p_v[X_t=w]-1/n\right)\dd t=:H^\V(v,w).
\end{equation}
To show~\eqref{eq:key_Angelo}, note first of all that $H^\V\1=0$. So it remains to show that $H^\V=L^{-1}$ on $W=\{f:f\perp \1\}$. First of all note that for $f\in W$
\begin{equation*}
    H^\V f(v) = \int_0^\infty \left[\sum_{w\in \V} \p_v(X_t=w)f(w) - \frac{1}{n} \sum_{w\in \V} f(w)\right] \dd t
    = \int_0^\infty S_tf(v)\dd t,
\end{equation*}
where $S_t=\exp(-tL)$ is the semigroup corresponding to $X$. In particular 
\begin{equation*}
    \sum_{v\in\V} H^\V f(v) = \int_0^\infty \sum_{v\in \V} S_tf(v)\dd t=0
\end{equation*}
since by symmetry of the random walk
\begin{equation*}
    \sum_{v\in \V} S_tf(v) = \sum_{v\in \V}\sum_{w\in \V} \p_v(X_t=w)f(w) = \sum_{w\in \V} f(w) \sum_{v\in \V} \p_w(X_t=v) = \sum_{w\in \V} f(w) = 0.
\end{equation*}
This implies that $H^\V$ maps $W$ into $W$. Moreover for $f\in W$
\begin{equation*}
    LH^\V f = L\int_0^\infty S_tf \dd t = \int_0^\infty LS_tf\dd t = \int_0^\infty \frac{\dd}{\dd t} \left(-S_t f\right) \dd t = -S_t f|_0^\infty = f.
\end{equation*}
Note that we used that $\lim_{t\rightarrow\infty} S_tf = 0$, since $f$ is zero-average. This finishes the proof.
\end{proof}

%From this point onward we leave the general setting of random graphs $\V$ and specialize our constructions to graphs $V_N$ made by $N$ points of $M$, which are those we are going to use later to show the scaling limit result. 
Now suppose our graph $\V$ consists of points of a manifold (which we generally denote by $p$ or $q$).
To speak of convergence of the DGFF to the GFF, we need to define them as comparable objects. To this end, we interpret them as random linear functionals on $W$. For the DGFF $\phi_\V$ this means introducing the following definition.
\begin{defi}[DGFF as random distribution]\label{def:distr}
Define for $f\in W$
\begin{equation*}
    \langle \phi_\V,\,f\rangle: = {\frac{1}{|\V|}}\sum_{p\in\V}f(p)\phi_\V(p).
\end{equation*}
%where $\phi_\V$ is the DGFF on $\V$ defined in Definition~\ref{defi:multivariate}.
\end{defi}
Note that, for each $\omega$ in the underlying probability space, $\phi_\V(\omega)$ is a well-defined linear functional on $W$, so an element of $W'$. Moreover, it is easy to see that this mapping is continuous (with respect to the weak* topology on $W'$), so in particular measurable. This implies that $\phi_\V$ can be interpreted as a random distribution on $M$.

\subsection{The continuum GFF}
Recall $W=\{f\in C^\infty(M): \int_M f\dd \overline V = 0\}.$ 
%Following \citet[pg. 523]{Sheffield:2007} we define the Dirichlet inner product ({Dirichlet energy}) on $W$ as
%\[
%(f,\,g)_\nabla:=\int_M \langle \nabla f,\,\nabla g\rangle\dd \overline V.
%\]
We now give the following definition.
\begin{prop}[GFF on $M$]
There exists a centered Gaussian random distribution $\phi:=\{\langle \phi,\,f\rangle:\,f\in W\}$ on $W'$ with covariance kernel $G$ given in~\eqref{eq:decom_G2}, that is, for all $f,\,g\in W$,
\begin{align*}
\E\left[\langle \phi,\,f\rangle\langle \phi,\,g\rangle\right]=(f,\,G g).
   % \mathcal L_\varphi:\,W&\rightarrow\R\\
   % f&\mapsto \exp\left(-\frac{1}{2}\|f\|^2_\nabla\right)
\end{align*}
%\int_M \left<\grad f,\grad f\right>\dd\bar V)$. 
We call this distribution the GFF on $M$.
\end{prop}
\begin{proof}
Note that $W$ is a nuclear space, being a subspace of the nuclear space $C^\infty(M)$. By the Bochner--Minlos theorem for nuclear spaces \cite[Theorem A]{umemurathesis:1965}, it suffices to show that the characteristic functional
\begin{align*}
  \mathcal L_\varphi:\,W&\rightarrow\R\\
    f&\mapsto \exp\left(-\frac{1}{2}(f,\,Gf)\right)
    \end{align*}
is continuous around $0$, positive definite and satisfies $\mathcal L_\varphi(0)=1$. The latter is clear. \\
To show positive definiteness one can use \citet[Proposition~2.4]{LSSW:2016}, which says that $L_\varphi$ is positive definite if $$f\mapsto (f,\,f)_G:=(f,\,G f)$$
is an inner product on $W$. This follows from the fact that $G$ is a self-adjoint positive definite operator on $W$ (compare~\eqref{eq:decom_G2}). Finally, since $f=G\rho\in L^2(M)$ is the unique solution with integral zero to the Poisson equation with input datum $\rho\in L^2(M)$, also with integral zero, one can use the Poincar\'e inequality and
\[
\|\nabla f\|^2_2=(\Delta_M f,\,f)\le\|f\|_2\|\rho\|_2
\]
to conclude that $G$ is a bounded and hence continuous operator on the set of zero-average square integrable functions on $M$. Since convergence in $W$ implies convergence in $L^2$, it is immediate to see with Cauchy-Schwarz that $f\mapsto (f,\,Gf)$ is continuous and hence that $\mathcal L_\varphi$ is continuous.
%Linearity and positiveness of $(\cdot,\cdot)_G$ are immediate, since $G$ is positive-definite (cf.~\eqref{eq:decom_G2}). Now suppose $(f,f)_G=0$. Then $\left<\grad f,\grad f\right>=0$ everywhere, so $\grad f=0$ everywhere, and thus $f$ is constant. Since $\int_M f=0$, we conclude that $f=0$.\\
%It remains to show continuity of $L_\phi$ around $0$. It suffices to show that $(f_n,f_n)_G$ goes to $0$ when $f_n$ goes to zero. Let $U$ be a compact subset of $M$ that is contained in a chart. Then on $U$ we have, for each $h\in W$, that $\grad h= g^{ij}\partial_i h\partial_j$, so
%\begin{equation*}
%    \left<\grad h,\grad h\right> = \left<g^{ij}\partial_i h\partial_j,g^{kl}\partial_k h\partial_l\right> = g^{ij}g^{kl}\partial_i h\partial_k h g_{jl} = \delta^i_l g^{kl} \partial_i h\partial_k h = g^{ik} \partial_i h\partial_k h.
%\end{equation*}
%Therefore we see on $U$ that $\left<\grad f_n,\grad f_n\right> = g^{ij}\partial_i f_n \partial_j f_n$. Since $f_n$ goes to $0$ in $W$, we know that each $\partial_i f_n$ goes to $0$ uniformly on compact sets contained in charts, so in particular on $U$. Moreover each $g^{ij}$ is continuous and hence bounded. Together this shows that $\left<\grad f_n,\grad f_n\right>$ goes to $0$ uniformly on $U$. By covering $M$ with compact sets that are contained in charts, we conclude that $\left<\grad f_n,\grad f_n\right>$ goes to $0$ uniformly on $M$. This implies that $(f_n,f_n)_\grad$ goes to $0$.
\end{proof}
%\subsubsection{Dual spaces and Green's function as covariance kernel} We now pass to showing the relation between the GFF covariance and the harmonic Green's function following \citet[Section~4.2]{Dubedat:2009}. Define the Green's function $G$ by setting
%\begin{equation*}
%    Gf = \begin{cases} (-\Delta_M)^{-1}f & f\in W \\ 0 & f\in W^\perp \end{cases}.
%\end{equation*}
%We know that
%\[
%\mathcal L_\phi(f)=\E \exp(\im\left\langle\varphi,f\right\rangle).
%\]
%By duality, one can define the space $W'$ to be equipped with the norm
%\[
%\|f\|_{W'}:=\sup_{g\in W:\,\|g\|_\nabla=1}\langle f,\,g\rangle
%\]
%where $\langle\cdot,\,\cdot\rangle$ denotes the dual pairing.
%Since
%\[
%(f,\,g)_\nabla=(f,\,\Delta g)
%\]
%for all $f,\,g\in W$, by the Riesz representation theorem applied to the linear functional $T_f(\cdot):= (f,\,\cdot)$ it holds that
%\[
%\langle f,\,g\rangle=(\Delta^{-1}f,\,g)
%\]

\subsection{Comments on Assumptions~\ref{as:gap}-\ref{as:empirical}} Let $(V_N)_{N=1}^\infty$ be a sequence of finite subsets of the manifold $M$ with corresponding conductances $c^N_{pq}=c^N_{qp}\geq 0$ for $p,\,q\in V_N$ such that each $V_N$ is connected in the sense described in Subsection~\ref{sec:DGFF}. Throughout this paper we assume that $V_N$ consists of $N$ points, which we label $p_1^N,\,\ldots,\,,p_N^N\in M$.\footnote{This is not an essential requirement, it just makes our notation less involved. For instance, for some natural sequences of grids the amount of points in $V_N$ is $N^d$ where $d$ is the dimension of the ambient space. With some straightforward changes our results hold in those cases too.} 
Let $(L_N)_{N=1}^\infty$, $(G_N)_{N=1}^\infty$ and $(\phi_N)_{N=1}^\infty$ be the sequences of corresponding generators, Green's functions and zero-average discrete Gaussian free fields on $V_N$, respectively, and for each $N$ let $\{S_t^N,t\geq 0\}$ denote the semigroup on $V_N$ that is generated by $L_N$. Note that we can also interpret $\phi_N$ as a random function on $W'$, as we described in Definition~\ref{def:distr}.\\
Let us comment more on the necessity of Assumptions~\ref{as:gap}-\ref{as:empirical} of Theorem~\ref{thm:1}. First of all, as we discussed above, all eigenvalues of $L_N$ are non-negative and only one eigenvalue equals $0$. We denote the second smallest eigenvalue (or the spectral gap) by $\lambda^N_2$. Then we know that $\lambda^N_2>0$, so each spectral gap is positive. Assumption~\ref{as:gap} says that the spectral gaps are uniformly positive, i.e.
\begin{equation*}
    \inf_N \lambda^N_2 >0.
\end{equation*}
Without this condition what could happen is that the spectrum of the graph laplacian would eventually capture the $0$-eigenvalue of $-\Delta_M$ \added[id=ale]{(compare~\citet[Result 3]{VL/B/B:2008} for a case in which spectral convergence fails)}. In this case, we would not be anymore in the domain of invertibility of the Green's function.
Secondly, we define the zero-average discrete version of any function $f:M\rightarrow\R$ to be
\begin{align*}
    f_N:V_N&\rightarrow \R\\
    p_i^N&\mapsto f(p_i^N) - \frac{1}{N}\sum_{i=1}^N f(p_i^N).
\end{align*}
Moreover, we define an inner product on $\R^{V_N}$ by $(f,g)=\sum_{i=1}^Nf(p_i^N)g(p_i^N)$. 
Now Assumption~\ref{as:semigroup} states that for each $f\in W$
\begin{equation*}
    \lim_{N\to\infty}\frac{1}{N}(f_N,S_t^N f_N) = (f,\,S_t f).
\end{equation*}
Assumption~\ref{as:semigroup} is probably the most natural one would expect in a convergence-to-GFF-type result: as we will see, it implies that the bilinear forms induced by the Green's functions converge pointwise (see Equation~\eqref{eq:convGreen} for the precise statement). One can ensure this limit via a stronger result, namely the uniform convergence of the discrete laplacian to the continuum one. This will be our strategy in the proof of Theorem~\ref{thm:2}. 
Finally, the third Assumption makes sure that the empirical measures corresponding to the grids converge weakly to the uniform distribution on the manifold. Therefore summing over grid points approximates integrating over the manifold in the same way as discrete lattice sums in $\Z^d$ approximate integrals in $\R^d$.

%% file: 3main.tex
\section{Proofs}\label{sec:proof}
Here we present the proofs of our main results. In Subsection~\ref{subsec:thm1} we will show that Assumptions~\ref{as:gap}-\ref{as:empirical} entail the convergence of the rescaled DGFF to the continuum one. We will show, using a spectral decomposition, that the variance of the distribution $\phi_N$ tested against smooth functions converges to that of the continuum field under Assumption~\ref{as:semigroup}. Assumptions~\ref{as:gap} and~\ref{as:empirical} will ensure enough regularity to get this convergence. Note that we will not use here the potential theory for the random walk to prove the scaling limit, in contrast to the $\Z^d$ case (a proof in $d=2$ is for example carried out in \citet[Section~1.4]{Biskup:2017}).

Theorem~\ref{thm:2} will be shown in Subsection~\ref{subsec:2}. We will sample uniform points from the manifold, and choose as conductances the heat kernel as explained in the Introduction. The proof of the validity of Assumptions~\ref{as:gap}-\ref{as:empirical} is in three steps (each step shows one assumption). First we will use the fact that the empirical measures corresponding to the grids almost surely converge in Kantorovich sense to the uniform measure $\overline V$, which implies weak convergence. Then we will show that the graph laplacians converge, uniformly over the grid points, to the Laplace--Beltrami operator. This will be done by choosing the bandwidth $t_N$ appropriately, following the ideas of \cite{vanGinkel:2017,vanGinkel/Redig:2018} (here we will need again the Kantorovich convergence of the empirical measures).  Finally, to show the bound on the spectral gap, we will use techniques developed in~\cite{VL/B/B:2008, Belkin/Niyogi:2007} by proving  convergence to an ``intermediate'' operator whose eigenvalues approximate those of the Laplace--Beltrami. This will yield a second condition on the rate of growth of $t_N$, and by combining the two we will obtain the final result.

\input{1proof.tex}
\input{2proof.tex}

%% file: 1proof.tex
\subsection{Proof of Theorem~\ref{thm:1}}\label{subsec:thm1}
We would like to prove that $\sqrt N\phi_N\rightarrow\phi$ in law in $W'$. Since $W$ is a nuclear Fr\'echet space, by~\citet[Theorem 2]{meyer9theoreme} 
it suffices to prove pointwise convergence of the characteristic functional, i.e.~that for any $f\in W$ 
\begin{equation*}
    \E \exp\left(\im\left\langle \sqrt N\varphi_N,f\right\rangle\right)\rightarrow  \E \exp(\im\left\langle\varphi,f\right\rangle).
\end{equation*}

Recall that we define $f_N:V_N\rightarrow\R$ by $f_N:=f|_{V_N}-{1}/{N}\sum_{i=1}^Nf|_{V_N}(p_i^N)$. We could subtract any constant from $f|_{V_N}$ since $\phi_N$ has average $0$, but we choose to subtract the discrete average since it ensures that $f_N$ belongs to the discrete counterpart of $W$. We can abbreviate
\[
G_N(i,\,j):=G_N(p_i^N,\,p_j^N)
\]
and we see that
\begin{align*}
    \E&\exp\left(\im\left\langle \sqrt N\varphi_N,f\right\rangle\right) = \E\exp\left(\im\frac{1}{N}\sum_{i=1}^Nf(p_i^N)\sqrt N\phi_N(p_i^N)\right) \\
    &= \exp\left(-\frac{1}{2N}\sum_{i=1}^N f_N(p_i^N)f_N(p_j^N) G_N (i,\,j)\right)=\exp\left(-\frac{1}{2N}(f_N,\,G_N f_N)\right).
\end{align*}
Further $\E\exp(\im\left\langle \varphi,f\right\rangle)=\exp(-{1}/{2}(f,\,Gf))$ (by definition of $\phi$). Therefore it suffices to show that 
\begin{equation}\label{eq:convGreen}
    N^{-1}\sum_{i=1}^N f_N(p_i^N)f_N(p_j^N) G_N (i,\,j)\rightarrow (f,\,Gf)
\end{equation}
for every $f\in W$. 

We now want to make use of the spectral decomposition of the Green's function. Let $0<\lambda_2^N\le\lambda_3^N\le\ldots\le\lambda_N^N$ be the non-zero eigenvalues of $L_N$. Define the measure $\mu^f_N$ on $\sigma(L_N)$ by
\begin{equation*}
   \mu^f_N(A) :=\sum_{j=2}^N \1_A(\lambda_i^N) \|P_{j,N}f_N\|_{2,N}^2
\end{equation*}
for $A\subset \sigma(L_N$). The total mass of the measure $\mu^f_N$ is 
\[\sum_{j=2}^N \|P_{j,N}f_N\|_{2,N}^2 = \|f_N\|^2_{2,N}.\]
Similarly define $\mu^f(A)$ on $\sigma(L)$ by
\begin{equation*}
    \mu^f(A):=\sum_{j=2}^\infty \1_A(\lambda_i) \|P_{j}f\|_2^2
\end{equation*}
for $A\subset \sigma(L)$ and $\lambda_2\le \lambda_3\le\ldots$ the positive eigenvalues of $-\Delta_M$. This is a measure with total mass 
\[\sum_{j=2}^\infty \|P_{j}f\|_2^2 = \|f\|^2_{2}.
\]
Note that since $P_{j,N}$ is a projection and since $P_{1,N}f_N=0$ by construction of $f_N$, we see by~(\ref{spect_rep_G_N}) that
\begin{align}
    (f_N,\,G_N f_N) &= \sum_{j=2}^N \frac{1}{\lambda_{j}^{N}} (f_N,P_{j,N}f_N)\nonumber \\
    &= \sum_{j=2}^N \frac{1}{\lambda_{j}^{N}} (P_{j,N}f_N,P_{j,N}f_N) = \sum_{j=2}^N \frac{1}{\lambda_{j}^{N}} \left\|P_{j,N}f_N\right\|^2_{2,\,N} = \E_{\mu^f_N}X^{-1}.\label{eq:X_inverse}
\end{align}
Analogously, by \eqref{eq:decom_G2} one deduces
\begin{equation*}
    (f,Gf) = \E_{\mu^f}X^{-1}.
\end{equation*}
%\textbf{Strategy 1 ---------------------------------------------------------------------------}\\
Now note that by Tonelli's theorem
\begin{equation*}
    \frac{1}{N}(f_N,\,G_N f_N) = \frac{1}{N}\E_{\mu^f_N}X^{-1} = \frac{1}{N}\E_{\mu^f_N}\int_0^\infty \e^{-tX}\dd t = \int_0^\infty \frac{1}{N}\E_{\mu^f_N} \e^{-tX}\dd t.
\end{equation*}
Denote $\delta:=\inf_N \lambda^N_2>0$ by Assumption~\ref{as:gap}. Then we see that
\begin{equation*}
    0\leq \frac{1}{N}\E_{\mu^f_N} \e^{-tX}\leq \frac{1}{N}\mu^f_N(\sigma(L_N)) \e^{-t\lambda_2^N} \leq \frac{1}{N}\|f_N\|^2_{2,N} \e^{-t\delta}.
\end{equation*}
Now note that
\begin{equation*}
    \frac{1}{N}\|f_N\|^2_{2,N} = \frac{1}{N} \sum_{i=1}^N \left(f(p_i^N)-\frac{1}{N}\sum_{j=1}^Nf(p_i^N)\right)^2 \leq \frac{1}{N} \sum_{i=1}^N f(p_i^N)^2.
\end{equation*}
By the continuity of $f$ and Assumption~\ref{as:empirical}, the last term converges to $\|f\|_2^2$. Therefore there exists a $C>0$ such that for all $N$
\begin{equation*}
    0\leq \frac{1}{N}\E_{\mu^f_N} \e^{-tX} \leq C \e^{-\delta t}.
\end{equation*}

By the dominated convergence theorem, this implies that
\begin{align}
    \lim_{N\rightarrow\infty} \int_0^\infty \frac{1}{N}\E_{\mu^f_N} \e^{-tX}\dd t &= \int_0^\infty \lim_{N\rightarrow\infty} \frac{1}{N} \E_{\mu^f_N} \e^{-tX}\dd t\nonumber\\
    &= \int_0^\infty \lim_{N\rightarrow\infty} \frac{1}{N}(f_N,\,S_t^Nf_N) \dd t.\label{eq:semigroup_appears}
\end{align}

Now we conclude thanks to Assumption~\ref{as:semigroup}:
\begin{align*}
    \lim_{N\rightarrow\infty}\frac{1}{N}(f_N,G_Nf_N) &\stackrel{\eqref{eq:X_inverse}}{=} \lim_{N\rightarrow\infty}\frac{1}{N}\E_{\mu^f_N}X^{-1} \stackrel{\eqref{eq:semigroup_appears}}{=} \int_0^\infty \lim_{N\rightarrow\infty} \frac{1}{N}(f_N,\,S_t^Nf_N) \dd t \\
    &= \int_0^\infty (f,\,S_tf) \dd t = (f,\,Gf).
\end{align*}
Note that in the last equality we have used the fact that $f$ has average zero on $M$.\qed
\begin{rmk}[Compatibility with known grids]\label{rmk:torus}
For any integer $N\in \N$ consider the quotient space $\mathbb{S}_N:=\Z/N\Z$. A finite product of $d$ copies of $\mathbb{S}_N$ defines a discrete torus $\mathbb T_N^d$ of side-length $N$. This object is naturally connected to the $d$-dimensional (flat) torus $\mathbb T^d$ given by a product of $d$ copies of $\mathbb S^1$. The rescaled graph Laplacian $L_N$ on $N^{-1}\mathbb T_N^d$ is the sum of the Laplacians $\mathcal L_N$ on each discrete $N^{-1}\mathbb S_N$ component. More precisely, $\mathcal L_N$ is defined for any $f : N^{-1}\mathbb S_N \to\R$ by
the following difference operator:
\[
\mathcal L_N f(k):=\frac{N^2}{4\pi^2}\left[(f(k)-f(k-1/N))+(f(k)-f(k+1/N))\right],\quad k\in N^{-1}\mathbb S_N.
\]
The spectra of $\mathcal L_N$ and $L_N$ are thus given by
\begin{align*}
\sigma(\mathcal L_N)&=\left\{\frac{N^2}{\pi^2}\sin^2\left(\frac{\pi k}{N}\right):\,k\in\{0,\,1,\,\ldots,\,N-1\}\right\},\\
\sigma(L_N)&=\left\{\frac{N^2}{\pi^2}\sum_{i=1}^d\sin^2\left(\frac{\pi k_i}{N}\right):\,k_i\in\{0,\,1,\,\ldots,\,N-1\},\,i\in\{1,\,\ldots,\,d\}\right\}.
\end{align*}
One can show that, with the rescaling $N^2$, the eigenvalues of $L_N$ converge to those of the Laplace--Beltrami operator on $\mathbb T^d$ as $N$ grows. Since the spectral gap of the Laplace-Beltrami operator is strictly positive, this ensures Assumption~\ref{as:gap}. A Taylor expansion yields that
\[
\mathcal L_N f(k)=\frac{f''(k)}{4\pi^2}+O\left(N^{-1}\right)
\]
and the $O$-term can be bounded uniformly in $k$ due to the compactness of the torus and the translation invariance of the situation. By summing over $d$ coordinate directions, we obtain the approximation to the Laplace--Beltrami operator on $\mathbb T^d$ (which is simply the sum of the second derivatives). A theorem of Trotter and Kurtz gives convergence of the corresponding semigroups, after which Assumption~\ref{as:semigroup} follows from a direct computation (see Corollary~\ref{cor:convSem} and Proposition~\ref{prop:a2holds} for the details in the manifold case). Finally, Assumption~\ref{as:empirical} is a consequence of the approximation of integrals via Riemann sums in $\R^d$.
\end{rmk}

%% file: 2proof.tex
\subsection{Proof of Theorem~\ref{thm:2}}\label{subsec:2}
Since the proof of Theorem~\ref{thm:2} is divided into three steps, the next three paragraphs will be dedicated to showing the validity of each assumption separately. 
%The third Assumption is the quickest one to check, since one can show that for uniformly sampled grid points the empirical distribution converges to the uniform distribution in the Wasserstein--Kantorovich topology. We will therefore begin with it, and go somehow in reverse order: the next one to be shown is Assumption~\ref{as:semigroup}, whose proof is carried through by showing the uniform convergence of the graph laplacian to $-\Delta_M$. The convergence will be obtained for a sequence of bandwidths $t'_N$. To ensure that also Assumption~\ref{as:gap} holds, we will need to adjust the $t'_N$ to new bandwidths $t_N$ such that both the convergence of the laplacians and of the spectral gap hold. 
\begin{rmk}[Quenched results]\label{rem:almostsure}
Note that all the upcoming assertions and quantities like the bandwidths depend on the realization of $(p_i^N)_{i=1}^N$. We will show a quenched result, meaning that we assume from now on that the grid points are fixed on $M$. Thus all the statements of this Subsection are meant in an almost-sure sense in the law of the grid points.
\end{rmk}

%%%%%%%
%%%%% EMpirical measures convergence
%%%%%
\subsubsection{Assumption~\ref{as:empirical} holds}
This Assumption, in the case of uniformly sampled grid points, is bypassed by the following stronger convergence result.
\begin{lemma}
Let $(p_i^N)_{i=1}^N$ be a sequence of i.i.d. points sampled from the normalized volume measure on $M$ and let $\mu^N$ be the corresponding empirical measure. Then
\[
\lim_{N\to \infty}W_1(\mu^N,\,\overline V)=0.
\]
\end{lemma}
\begin{proof}
The proof follows~\citet[Example 5.15]{handel2016prob} and the details for the manifold case can be found in~\citet[Section~3.4, ``Convergence of a random grid'']{vanGinkel/Redig:2018}. We give a brief sketch of the proof here.

First one shows that $\E [W_1(\mu^N,\,\overline V)]$ goes to $0$. To do so one needs to control
\begin{equation}\label{eq:kant}
    W_1(\mu^N,\,\overline V)=\sup_{f\in\mathscr{F}_1} \left(\int f\dd \mu^N - \int f\dd \overline V\right),
\end{equation}
where $\mathscr{F}_1$ denotes the set of Lipschitz functions on $M$ with Lipschitz constant at most $1$. Realizing that~\eqref{eq:kant} is invariant under addition of constants and that Lipschitz functions on a compact manifold are bounded, the proof can be reduced to functions taking values in $[0,K]$ for some $K$. We denote these Lipschitz functions by $\mathscr{F}_{1,K}$. Viewing~\eqref{eq:kant} as a supremum over random variables, one can apply~\citet[Lemma 5.7]{handel2016prob} to bound the expectation of~\eqref{eq:kant} by a covering number:
\begin{equation*}
    \E [W_1(\mu^N,\,\overline V)]\leq \inf_{\epsilon>0}\left\{ 2\epsilon+\sqrt{\frac{2K^2}{N} \log N(\mathscr{F}_{1,K},||\cdot||_\infty,\epsilon)}\right\}.
\end{equation*}
Here $N(\mathscr{F}_{1,K},||\cdot||_\infty,\epsilon)$ is the size of the smallest $\epsilon$-net in $\mathscr{F}_{1,K}$, where an $\epsilon$-net $\mathscr{G}$ is a set of functions such that every function in $\mathscr{F}_{1,K}$ is at most $\epsilon$ away from a function in $\mathscr{G}$ in the uniform distance.

Next, one proves that $N(\mathscr{F}_{1,K},||\cdot||_\infty,\epsilon)$ can be bounded by $\exp(c/\epsilon^d)$ for some constant $c$ and for $\epsilon$ small enough. This is done by constructing an $\epsilon$-net in an efficient way. 
Then, optimizing over $\epsilon$ and sending $N$ to infinity, one sees that $\E [W_1(\mu^N,\,\overline V)]$ goes to $0$.

Finally, one can use a concentration inequality (\citet[Theorem 3.11]{handel2016prob}) to show that the probability that $W_1(\mu^N,\,\overline V)$ deviates more than a constant from $0$ decays exponentially fast. With Borel--Cantelli, this implies almost sure convergence of $W_1(\mu^N,\,\overline V)$ to $0$.
\end{proof}

%%%% Unif convergence laplacians
%%%%%%%%
\subsubsection{Assumption~\ref{as:semigroup} holds}\label{subsubsec:as2holds} This Subsection is based on proving one key Proposition:
\begin{prop}\label{prop:unif_Delta_conv}
Set the bandwidth parameter $t'_N$ to satisfy~\eqref{eq:Wass_bound}. Then the graph laplacian $L_N$ on $V_N$ is such that for all $f\in W$ the following holds:
\begin{equation*}
    \lim_{N\to\infty}\left\|L_N f|_{V^N} - \big((-\Delta_M)f\big)|_{V_N}\right\|_{\infty,N}= 0.
\end{equation*}
\end{prop}
In order to prove Proposition~\ref{prop:unif_Delta_conv} we begin with a few remarks based on the approach of~\citet[Section~3.2]{vanGinkel/Redig:2018}, which we recall here for completeness. Choose $i\in\{1,\,\ldots,\,N\}$. We see that 
\[
-L_N f(p_i^N)=\int_M g^{t'_N,\,i}(p) \mu^N(\dd p)
\]
where 
\[
g^{t'_N,\,i}(p):=\frac{p_{t'_N}(p,\,p_i^N)}{t'_N}(f(p)-f(p_i^N)),\quad p\in M.
\]
To avoid cumbersome notation we will now drop the $N$ sub/superscript in $t'_N$ and $p_i^N$. It is clear that one can write
\begin{equation}\label{eq:break_lapl}
-L_N f(p)=\int_M g^{t',\,i}(p) \overline V(\dd p)+\int_M g^{t',\,i}(p) (\mu^N-\overline V)(\dd p).
\end{equation}
The strategy of the proof consists in showing that the first term converges to $(-\Delta_M) f$, and the second one becomes negligible in the limit $N\to\infty$. To this purpose, we need a bound on the supremum norm and the Lipschitz constant of the heat kernel. In the following we use $L_f$ to denote the Lipschitz constant of a function $f$.
\begin{lemma}\label{lem:bound_Lip_HK}
For $t$ small enough one has
\[
\sup_{x,\,y\in M}|p_t(x,\,y)|\le C t^{-\frac{d}{2}}
\]
and
\[
\sup_{x,\,y\in M}L_{p_t(x,\,y)}\le C t^{-\frac{d}{2}-1}
\]
where $C$ depends only on the curvature of the manifold and on the dimension.
\end{lemma}
\begin{proof}
Let us first recall the classical Gaussian bound on the heat kernel \cite[Corollary 3.1]{Li/Yau:1986}: 
\begin{equation}\label{eq:boundLY}
p_t(x,\,y)\le C \frac{\e^{-\frac{d^2(x,\,y)}{C t}+C K t}}{\sqrt{V(x,\,\sqrt t)V(y,\,\sqrt t)}}
\end{equation}
where $K\ge 0$ is such that $\Ric(M)\ge -K$ and where $V(x,\,r)$ denotes the volume of the ball around $x\in M$ with radius $r>0$ in the geodesic distance. Note that such $K$ exists in our situation, since $M$ is compact. A simple argument (comparing with a space of constant curvature) shows that there is a $C>0$ that does not depend on $x$ such that $\inf_{x\in M}V(x,\,\sqrt t)\ge C t^{d/2}>0$ for every $x$ when $t$ is small enough. This immediately entails the sup-norm bound for the function $p_t(\cdot,\,\cdot)$.
As far as the gradient is concerned, we use the bound in \citet[Theorem 1]{Engoulatov:2006} to deduce that
\begin{align}
\nabla p_t(x,\,y)&=\nabla \log p_t(x,\,y)\cdot p_t(x,\,y)\nonumber\\
&\stackrel{\eqref{eq:boundLY}}{\le} C(R,\,d)\left(\frac{D}{t}+\frac1{\sqrt t}+K\sqrt t\right)\left(\frac{\e^{-\frac{d^2(x,\,y)}{C t}+C K t}}{\sqrt{V(x,\,\sqrt t)V(y,\,\sqrt t)}}\right)\label{eq:bound_grad_HK}
\end{align}
and $D:=\mathrm{diam}(M)<\infty.$ Bounding the exponential term by an absolute constant and plugging this in \eqref{eq:bound_grad_HK} one obtains that
\[
\nabla p_t(x,\,y)\le C \left(\frac{D}{t}+\frac1{\sqrt t}+R\sqrt t\right)t^{-\frac{d}{2}}
\]
which concludes the proof.
\end{proof}
This entails easily that the second summand on the right-hand side of~\eqref{eq:break_lapl} goes to zero as $t'$ goes to zero, namely one can derive the following.
\begin{cor}\label{cor:second_term_0} Uniformly over $i\in\{1,\,\ldots,\,N\}$ one has
\[
\lim_{t\to 0}\left|\int_M g^{t',\,i}(p) \dd (\mu^N-\overline V)(\dd p)\right|=0.
\]
\end{cor}
\begin{proof}
Observe that
\[
L_{g^{t',\,i}}\le \frac{1}{t'}\left(L_{p_{t'}(\cdot,\,p_i)}\|f(\cdot)-f(p_i)\|_\infty+\|p_{t'}(\cdot,\,p_i)\|_\infty L_{f(\cdot)-f(p_i)}\right).
\]
Note that $L_f<\infty$ exists since $f$ is smooth and that $L_{f(\cdot)-f(p_i)}=L_f$ since $f(p_i)$ is a constant. Therefore 
\begin{align*}
&\left|\int_M g^{t',\,i}(p)  (\mu^N-\overline V)(\dd p)\right|\\
&\le \frac{1}{t'}\left(L_{p_{t'}(\cdot,\,p_i)}\|f(\cdot)-f(p_i)\|_\infty+\|p_{t'}(\cdot,\,p_i)\|_\infty L_{f}\right) W_1(\mu^N,\,\overline V)\\
&\le \frac{C}{t'}\left((t')^{-\frac{d}{2}-1}C\|f\|_\infty+(t')^{-\frac{d}{2}} L_{f}\right) W_1(\mu^N,\,\overline V)
\end{align*}
where in the last line we have used Lemma~\ref{lem:bound_Lip_HK}. The conclusion is a consequence of~\eqref{eq:Wass_bound}. Uniformity follows since the bounds do not depend on $i$.
\end{proof}
We can now begin with the proof of Proposition~\ref{prop:unif_Delta_conv}.
\begin{proof}[Proof of Proposition~\ref{prop:unif_Delta_conv}] Considering the break-up of the graph laplacian as in~\eqref{eq:break_lapl} and Corollary~\ref{cor:second_term_0} (remember that $t'=t'_N$ is infinitesimal as $N$ grows), all that is left to show is that
\[
\lim_{N\to\infty}\sup_{1\le i\le N}\left|(-\Delta_M) f(p_i)-\int_M g^{t',\,i}(p)\overline V(\dd p)\right|=0.
\]
Now observe that
\[
\int_M g^{t',\,i}(p)\overline V(\dd p)=-\left(\frac{\1-S_{t'}}{t'}f\right)(p_i).
\]
Since $\Delta_M$ generates $(S_t,t\geq 0)$, we know for any smooth $f$ that
\[
\left(\frac{S_{t'}-\1}{t'}f\right)(p) \rightarrow \Delta_M f(p)
\]
uniformly in $p\in M$ as $t'$ goes to $0$ (see for instance~\citet[Theorem 7.13]{grigoryan2009heat}), so in particular uniformly in the $p_i$'s. Since $t'$ goes to $0$ as $N$ goes to infinity, this concludes the proof.
\end{proof}

As a consequence we obtain the following. 
\begin{cor}\label{cor:convSem}For all $t>0$ and $f\in W$
\begin{equation*}
    \lim_{N\to\infty}\|S_t^N f|_{V_N} - (S_tf)|_{V_N}\|_{\infty,N}= 0.
\end{equation*}
\end{cor}
\begin{proof}
The proof is a direct application of Theorem 2.1 from~\cite{kurtz1969extensions} and Proposition~\ref{prop:unif_Delta_conv}, combined with an argument that the extended limit of $L_N$ (as defined in Kurtz's paper) equals the Laplace-Beltrami operator. The reason is that they are both generators and they agree on the set of smooth functions (by Proposition~\ref{prop:unif_Delta_conv} they agree on $W$ and it is easy to see that they are both $0$ on constant functions), which forms a core for the Laplace-Beltrami operator.
\end{proof}
%%%%%%%%
%%%%% Lemma on semigroup convergence. where to put it?
We are now ready to show Assumption~\ref{as:semigroup}.
\begin{prop}\label{prop:a2holds}
For all $f\in W$, Assumption~\ref{as:semigroup} holds.
\end{prop}

\begin{proof}
Denote $f|_{V_N}$ by $f|_N$ and ${1}/{N}\sum_{i=1}^N f(p_i)$ (both the constant and the constant function) by $\bar f^N$. Then $f_N=f|_N-\bar f^N$, which implies that
\begin{equation}\label{eq:decompo}
    (f_N,S_t^Nf_N) = (f_N,S_t^Nf|_N)-(f|_N,S_t^N\bar f^N) +(\bar f^N,S_t^N\bar f^N).
\end{equation}
Since $\bar f^N$ is constant, $S_t^N\bar f^N = \bar f^N$. Thus we see for the second summand above that
\begin{equation*}
\frac{1}{N}(f|_N,S_t^N\bar f^N)= \frac{1}{N}(f|_N,\bar f^N) = \frac{1}{N}\sum_{i=1}^N f(p_i)\frac{1}{N}\sum_{j=1}^Nf(p_j)\rightarrow \int_M f\dd\bar V\int_M f\dd\bar V=0.
\end{equation*}
For the same reason, we see for the third summand in~\eqref{eq:decompo}
\begin{equation*}
\frac{1}{N}(\bar f^N,S_t^N\bar f^N)= \frac{1}{N}(\bar f^N,\bar f^N) %= \frac{1}{N}\sum_{k=1}^N\left(\sum_{i=1}^N f(p_i)\sum_{j=1}^Nf(p_j)\right) = \sum_{i=1}^N f(p_i)\sum_{j=1}^Nf(p_j) 
\rightarrow \int_M f\dd\bar V\int_M f\dd\bar V=0.
\end{equation*}
Now we deal with the first summand of the right-hand side of~\eqref{eq:decompo}:
\begin{equation}
    (f_N,S_t^Nf|_N) = (f_N,(S_tf)|_N)+(f_N,S_t^Nf|_N-(S_tf)|_N)\label{eq:decompo2}
\end{equation}
The first term gives
\begin{eqnarray*}
    \frac{1}{N}(f_N,(S_tf)|_N) &=& \frac{1}{N}(f|_N,(S_tf)|_N)-\frac{1}{N}(\bar f^N,(S_tf)|_N) \\
    &=& \frac{1}{N} \sum_{i=1}^N f(p_i)S_tf(p_i) - \frac{1}{N}\sum_{i=1}^N \bar f^N S_tf(p_i) \\
    &\longrightarrow& \int_M fS_tf \dd \bar V - \int_M f \dd V \int_M S_tf \dd V = (f,S_tf)-0.
\end{eqnarray*}
Now we need to show that the last term in the right-hand side of~\eqref{eq:decompo2} goes $0$. Note that
\begin{equation}\label{eq:decompo3}
    |(f_N,S_t^Nf|_N-(S_tf)|_N)|\leq \sum_{i=1}^N |f_N(p_i)| ||S_t^Nf|_N-(S_tf)|_N||_{N,\infty}.
\end{equation}
Recall that $||S_t^Nf|_N-(S_tf)|_N||_{N,\infty}\rightarrow 0$ by Corollary~\ref{cor:convSem}. Moreover,
\begin{equation*}
    \frac{1}{N}\sum_{i=1}^N |f_N(p_i)|\leq \frac{1}{N}\sum_{i=1}^N|f(p_i)| + \left|\bar f^N\right| \rightarrow \int_M |f|\dd \bar V + |\int f\dd \bar V| = \int |f|\dd \bar V<\infty.
\end{equation*}
Combining these results with~\eqref{eq:decompo3} yields
\begin{equation*}
    \limsup_{N\rightarrow\infty} \frac{1}{N}|(f_N,S_t^Nf|_N-(S_tf)|_N)|\leq \int |f|\dd \bar V\cdot 0=0.
\end{equation*}
We conclude that $(f,\, S_t f)$ is the only non-zero remaining term when taking the limit $N\to\infty$ in~\eqref{eq:decompo}, which was to be shown.
\end{proof}
%%%
%%% The spectral gap assumption
%%%%%
\subsubsection{Assumption~\ref{as:gap} holds} \added[id=bart]{For this proof, we denote the graph laplacian as $L_N^t$, thus now highlighting the dependence on both $N$ and $t$:}
\begin{equation*}
    L_N^tf(v) = -\sum_{w\in V_N} \frac{p_t(v,w)}{Nt}(f(w)-f(v)).
\end{equation*}
\added[id=bart]{The idea is that, by letting first $N$ to infinity and then $t$ to $0$,  we prove that the spectral gaps $\lambda_{N,2}^t$ of $L^t_N$ converge to the spectral gap of the Laplace--Beltrami operator, i.e.}
\begin{equation*}
    \lim_{t\rightarrow 0}\lim_{N\rightarrow\infty} \lambda_{N,2}^t = \lambda_2.
\end{equation*}
\added[id=bart]{From this we will extract a sequence $t_N$ such that the spectral gap of $L_N=L_N^{t_N}$ converges (i.e. $\lambda_{N,2}\rightarrow \lambda_2).$ We will show that this sequence can be constructed in such a way that the convergence of Assumption~\ref{as:semigroup} still holds.}
We will base our proof on the ideas employed by~\cite{Belkin/Niyogi:2007} to prove convergence of the graph laplacian eigenmaps to the continuum ones. In the article, the authors use the ``intermediate'' operator $L^t:\,L^2(M)\to L^2(M)$, $t> 0$,  defined via
\[
L^t f(p):=t^{-1}\int_M p_t(p,\,q)(f(p)-f(q))\overline V(\dd q)
\]
whose eigenvalues we denote by $\lambda_1^t\le \lambda_2^t\le\ldots$ In their case, the heat kernel edge weights were replaced by the Gaussian kernel in some Euclidean ambient space. Instead, with our choices note that
\[
L^t=\frac{\1-S_t}{t}.
\]
%and hence we immediately see that $\Delta_M=\lim_{t\to 0}t^{-1}(\1-S_t)$. Thus 
Therefore the $i$-th eigenvalue of $L^t$ equals $t^{-1}(1-\exp(-t\lambda_i))$, with $\lambda_i$ the $i$-th eigenvalue of the Laplace--Beltrami, so in particular we see
\begin{equation}
    \lim_{t\to 0}\lambda^t_2= \lambda_2 \label{convint}.
\end{equation}

%The strategy adopted by~\cite{Belkin/Niyogi:2007} to prove spectral convergence involved a perturbation operator $R^t:=t^{-1}(\1-S_t)-L^t$ that measured how far the functional approximation $L^t$ lied from the operator $t^{-1}(\1-S_t)$. In the case of Belkin and Niyogi this perturbation was necessary since the authors used as conductances the Gaussian kernel, and not the heat kernel. In our case, this is no longer necessary because our $L^t$ coincides with $t^{-1}(\1-S_t)$, and therefore~\citet[Theorem~4.1]{Belkin/Niyogi:2007} greatly simplifies giving~~\eqref{convint} immediately. 
Using~\citet[Theorem~21, Proposition~23]{VL/B/B:2008} analogously to what is done by~\citet[Theorem~3.2]{Belkin/Niyogi:2007} one also obtains that
\begin{align}
    \lim_{N\to\infty}\lambda^t_{N,2} =\lambda^t_2 &\quad\text{a.s. }\label{convinN}
\end{align}
Note that this is an almost sure result in the law of the grid points. Since the intersection of two probability one sets still has probability one, we can safely assume that for the grid that was fixed in Remark~\ref{rem:almostsure} the limit above holds.

Now we want to construct a sequence $(t_N)_{N=1}^\infty$ such that we can reduce ~\eqref{convint}-~\eqref{convinN} to one limit:
\begin{equation}\label{eq:conv_t_N}
    \lim_{N\to\infty}\lambda^{t_N}_{N,2}=\lambda_2%\quad\text{a.s.}
\end{equation}
We constructed the sequence $(t'_N)_{N=1}^\infty$ in Subsubsection~\ref{subsubsec:as2holds} to prove pointwise convergence of the Laplacians. It is direct from those calculations that any sequence that goes to $0$ more slowly than $(t'_N)_{N=1}^\infty$ would also suffice. Therefore we first construct $(t_N)_{N=1}^\infty$ to ensure \eqref{eq:conv_t_N} and such that $t_N\geq t'_N$ for each $N$, after which we can simply replace $t'_N$ in Subsubsection~\ref{subsubsec:as2holds} by $t_N$.

\begin{lemma}
There exists a sequence $(t_N)_{N=1}^\infty$ such that the following hold: 
\begin{itemize}[label=\raisebox{0.25ex}{\tiny$\bullet$}]
    \item $t_N\downarrow 0$ as $N\to\infty$,
    \item $\lim_{N\to\infty}\lambda^{t_N}_{N,2}= \lambda_2,$
    \item $t_N\geq t'_N$ for every $N\in\N.$
\end{itemize}
\end{lemma}
\begin{proof}
For $j\in\N$ choose $n_j$ such that:
\begin{enumerate}[label=(\roman*)]
    \item $n_j>n_{j-1}$ for $j\geq 2$,
    \item $|\lambda^{1/j}_{n,2}-\lambda^{1/j}_2|\leq {1}/{j}$ for all $n\geq n_j$,
    \item $n_j\geq \min\{k\in\N:t'_k\leq 1/j\}$.
\end{enumerate}
Such $n_j$ exists because of~(\ref{convinN}) and because $t'_N\rightarrow 0$. Now for $N\in\N$ define $j(N)\in\N$ such that
\begin{equation*}
    n_{j(N)}\leq N < n_{j(N)+1}
\end{equation*}
and set
\begin{equation*}
    t_N:= \frac{1}{j(N)}.
\end{equation*}
First of all $j(N)$ is well-defined for each $N$ because of (i). Moreover, we directly see that $j(N)\uparrow\infty$, so $t_N\downarrow 0$. Note that it follows from (iii) and the fact that $t'_N$ is decreasing that $t'_{n_j}\leq 1/j$. Using this and the monotonicity of $t'_N$, we see
\begin{equation*}
    t_N=\frac{1}{j(N)}\geq t'_{n_{j(N)}} \geq t'_N.
\end{equation*}
We also see
\begin{equation*}
    |\lambda^{t_N}_{N,2}-\lambda_2|\leq \underbrace{|\lambda^{t_N}_{N,2}-\lambda^{t_N}_2|}_\text{=:(I)} + \underbrace{|\lambda^{t_N}_{2}-\lambda_2|}_\text{=:(II)}.
\end{equation*}
(II) goes to $0$ because of~(\ref{convint}) and the fact that $t_N\downarrow 0$. Further we see
\begin{equation*}
    \text{(I)} = \Big|\lambda^{1/j(N)}_{N,2}-\lambda^{1/j(N)}_2\Big|\leq \frac{1}{j(N)},
\end{equation*}
because of (ii) and the assumption $N\geq n_{j(N)}$ by construction. Since $1/j(N)\rightarrow 0$, the result follows.
\end{proof}

%% file: 4voronoi.tex
\section{Convergence of the Voronoi extension}\label{subsec:vor}
In this Section we would like to state and prove Theorem~\ref{thm:vor}. The proof consists of two main blocks: tightness in $H^{-s}(M)$ and finite-dimensional convergence.

We start with the necessary definitions. 
\subsection{Definitions} 
For $s\ge 0$ we define the space $H^s:=H^s(M)$ as the closure of $W$ with respect to the norm
\[
\|f\|_s^2:=\sum_{j=2}^\infty \lambda_j^s\|P_j f\|_2^2=\sum_{j=2}^\infty \lambda_j^s (f,e_j)^2,
\]
and the corresponding inner product
\begin{equation*}
    (f,g)_s = \sum_{j=2}^\infty \lambda_j^s (f,e_j)(g,e_j),
\end{equation*}
where $(e_j,j\geq 0)$ is an $L^2(M)$-orthonormal basis of eigenfunctions of the Laplace-Beltrami operator. Note that all the $e_j$'s are smooth. We denote by $H^{-s}$ the Hilbert space dual of $H^s$. 

We will need the following properties.
\begin{lemma}
Our definition of $H^s$ coincides with the usual definition of Sobolev space on $M$ (as described in for instance~\citet[Section~6]{canzani2013analysis} for $s>0$, for $s<0$ they are just the dual of $H^{-s}$). Moreover the canonical norm on $H^{-s}$ induced by $H^s$ satisfies 
\[
\|\psi\|_{-s}^2=\sum_{j=2}^\infty
\lambda_j^{-s}\langle \psi,\,e_j \rangle^2,\quad \psi\in H^{-s}.\]
\end{lemma}
\begin{proof}
The first statement follows from \citet[Proposition~56]{canzani2013analysis}. 

For all $\psi\in  H^{-s}$ by Riesz representation theorem, there exists $f_\psi\in H^s$ such that $\langle \psi, g\rangle= (f_\psi, g)_s$ for all $g\in H^s$. Also by isometry we have that $\|\psi\|_{-s}=\|f_\psi\|_s$. Now note that
\begin{equation*}
    \langle \psi, e_j\rangle = (f_\psi,e_j)_s=\sum_{k=2}^\infty \lambda_k^s (f_\psi,e_k)(e_j,e_k) = \lambda_j^s (f_\psi,e_j)
\end{equation*}
Hence we have 
\begin{equation*}
    \|\psi\|_{-s}^2= \|f_\psi\|_{s}^2= \sum_{j= 2}^\infty \lambda_j^{s} (f_\psi, e_j)^2 = \sum_{j=2}^\infty \lambda_j^{-s} \langle v, e_j\rangle^2.
\end{equation*}
\end{proof}
Furthermore, we will need the following classical result to prove tightness (its proof is analogous to~\citet[Theorem~5.8]{roe2013elliptic}).
\begin{thm}[Rellich's theorem] \label{sobolevprops} If $s<t$ then the inclusion operator $H^{t}\hookrightarrow H^{s}$ is compact.% (this is known as Rellich's theorem).
\end{thm}

Now let $\{C^N_i,i=1,\,\ldots,\,N\}$ be the Voronoi tessellation corresponding to the vertex set $V_N:=(p_i)_{i=1}^N$, i.e.
\begin{equation*}
    C^N_i=\{p\in M: d(p,p_i)\leq d(p,p_j) \hspace{2mm} \forall\, j\leq N\},\quad i=1,\,\ldots,\,N.
\end{equation*}
Also denote $v^N_i=\overline V(C^N_i)$. We will usually leave out the superscript $N$ to ease notation.
\begin{defi}[The DGFF in $H^{-s}$] Let $\phi_N$ be the zero-average DGFF on $V_N$ as in Theorem~\ref{thm:1}. We define $\widetilde \phi_N\in H^{-s}$ by the following action on $f\in H^s$:
\begin{equation*}
    \left\langle \widetilde \phi_N, f\right\rangle := \frac{1}{N}\sum_{i=1}^N \phi_N(p_i)\frac{1}{v_i}\int_{C_i}f(p)\overline V(\dd p).
\end{equation*}
\end{defi}
Note that if we define 
\begin{eqnarray*}
\widetilde f_N:\,V_N&\rightarrow& \R\\
p_i & \mapsto& \widetilde f_N(p_i):=\frac{1}{v_i}\int_{C_i}f(p)\overline V(\dd p)
\end{eqnarray*}
then we can write 
\begin{equation}\label{eq:tilde_useful}
    \langle \widetilde \phi_N,f\rangle=N^{-1}(\phi_N,\widetilde f_N)=\langle \phi_N,\widetilde f_N\rangle
\end{equation}
with a slight abuse of notation (since $\phi_N$ acts on $W$, but in fact this action depends only on grid values).
In order to prove Theorem~\ref{thm:vor} first of all we will show that the sequence $\{\widetilde \phi_N,\,N\in\N\}$ is tight in $H^{-s}$ (Subsection~\ref{subsec:tight_vor}). From this it follows that every sequence has a convergent subsequence. Then what remains is to show that the limit is unique. Since the limit is Gaussian, it is characterized by its finite-dimensional distributions. By the theory of abstract Wiener spaces, already described for example in~\citet[Section~3.2]{CDH18}, it suffices to show that for all $f,\,g\in H^1$
\begin{equation*}
    \E \left(\langle \sqrt N\widetilde \phi_N, f\rangle\langle\sqrt N \widetilde \phi_N, g\rangle\right) \to (f,\,Gg)
\end{equation*}
as $N\to\infty$. This will be done in Subsection~\ref{subsec:fdd_vor}.
\subsection{Tightness of \texorpdfstring{$\tphi$}{}}\label{subsec:tight_vor}
We prove the following Proposition.
\begin{prop}
The collection $\{\widetilde \phi_N,N\in\N\}$ is tight in $H^{-s}$ for any $s>d-1/2$.
\end{prop}
\begin{proof}
We will first prove that for $s>d-1/2$ and for every $\epsilon>0$, there exists $R=R(\epsilon)>0$ such that for all $N$
\begin{equation}\label{tightness}
    \p(\|\sqrt N\widetilde \phi_N\|^2_{-s}>R)\leq \epsilon.
\end{equation}
First of all by Chebyshev's inequality
\begin{eqnarray*}
    \p(\|\sqrt N\widetilde \phi_N\|^2_{-s}>R)\leq \frac{1}{R} \E(\|\sqrt N\widetilde \phi_N\|^2_{-s}).
\end{eqnarray*}
It suffices then to show that $\E(\|\sqrt N\widetilde \phi_N\|^2_{-k})$ is bounded by some constant.
We write
\begin{eqnarray*}
    \E\left(\sum_{j=2}^\infty \lambda_j^{-s} \langle \sqrt N \widetilde \phi_N,e_j\rangle^2\right) = \sum_{j=2}^\infty \lambda_j^{-s} \E\left(\langle \sqrt N \widetilde \phi_N,e_j\rangle^2\right).
\end{eqnarray*}
Now note that for any $h\in W$
\begin{equation}\label{eq:bound_Bangalore}
    \E\left(\langle \sqrt N\widetilde \phi_N,h\rangle^2\right) \stackrel{\eqref{eq:tilde_useful}}{=} \E\left(\langle \sqrt N \phi_N,\widetilde h_N\rangle^2\right) = \frac{1}{N}(\widetilde h_N,G_N\widetilde h_N)\leq \frac{1}{N} \|\widetilde h_N\|^2 \|G_N\|
\end{equation}
\added[id=ale]{where $\|G_N\|$ is the operator norm of $G_N$ from $\ell^2(V_N)$ to itself and $\|\widetilde h_N\|$ is the $\ell^2(V_N)$-norm.}
Since $\|G_N\|= (\lambda_2^N)^{-1}$ by Assumption~\ref{as:gap} we can bound it by some constant independent of $N$. Moreover
\begin{equation*}
    \frac{1}{N} \|\widetilde h_N\|^2 = \frac{1}{N}\sum_{i=1}^N \left(\frac{1}{v_i}\int_{C_i}h(p)\overline V(\dd p)\right)^2\leq \|h\|^2_\infty.
\end{equation*}
Now by~\citet[Theorem 82]{canzani2013analysis} $\|e_j\|_\infty\leq C\lambda_j^{(d-1)/4}$, \added[id=ale]{so applying the previous argument to the bound~\eqref{eq:bound_Bangalore} with $h:=e_j$ we see that} 
\begin{equation*}
    \sum_{j=1}^\infty \lambda_j^{-s} \E\left(\langle \sqrt N \widetilde \phi_N,e_j\rangle^2\right)
    \leq \sum_{j=1}^\infty \lambda_j^{-s} \|G_N\| \|e_j\|_\infty^2 \leq C\sum_{j=1}^\infty \lambda_j^{(d-1)/2-s}.
\end{equation*}
\citet[Theorem 72]{canzani2013analysis} states Weyl's lemma with the asymptotic $\lambda_j\sim C j^{2/d}$ as $j\rightarrow \infty$, which shows that 
\begin{equation*}
    C\sum_{j=1}^\infty \lambda_j^{(d-1)/2-s} \leq C \sum_{j=1}^\infty j^{2/d((d-1)/2-s)}.
\end{equation*}
This series is bounded as long as $2/d((d-1)/2-s) < -1$, so for $s>d-1/2$. This means we have shown~\eqref{tightness}.

To conclude the argument, fix $s>d-1/2$. Let $s'$ be such that $s>s'>d-1/2$ and let $\epsilon>0$. We know there exists $R>0$ such that~\eqref{tightness} holds, i.e. for all $N$
\begin{equation*}
    \p(\widetilde \phi_N \notin \overline{B_{-s'}(0,R)}) \leq \epsilon,
\end{equation*}
where $\overline {B_{-s'}(0,R)}$ is the closed ball with radius $R$ in $H^{-s'}$. Now by Theorem~\ref{sobolevprops}, we see that $\overline{B_{-s'}(0,R)}$ is compact in $H^{-s}$ (since $s>s'$), so we have shown tightness in $H^{-s}$.
\end{proof}

\subsection{Convergence of finite dimensional distributions}\label{subsec:fdd_vor}
As mentioned before, we need to show that for all $f,\,g\in H^1$
\begin{equation*}
    \E \left(\langle \sqrt N\widetilde \phi_N, f\rangle\langle \sqrt N\widetilde \phi_N, g\rangle\right) \to (f,\,Gg).
\end{equation*}
Since $W$ is dense in $H^{1}$ and by a polarization argument, it suffices to show the following.
\begin{prop}\label{prop:marginals}
For all $f\in W$
\begin{equation*}
    \E \left\langle\sqrt N \widetilde \phi_N,f\right\rangle ^2 \rightarrow (f,\,Gf).
\end{equation*}
\end{prop}
Before we move on to the proof, we prove the following technical lemma.
\begin{lemma}\label{epsilonlemma}
Define
\begin{equation*}
    \epsilon_N := \sup_{1\leq i\leq N} \sup_{p\in C_i^N} d(p,p_i).
\end{equation*}
Then $\epsilon_N$ goes to $0$ as $N\rightarrow \infty$.
\end{lemma}
\begin{proof}
To derive a contradiction, suppose that $\epsilon_N$ does not go to $0$. This means that there is some $\delta>0$ such that $\epsilon_N>2\delta$ for infinitely many $N$. Consequently for each such $N$, there exists $1\le i\leq N$ and $p\in C_i$ such that $d(p,\,p_i)\geq \delta$. Since $p\in C_i$, $p_i$ is the nearest grid point to it. This implies that $B(p,\delta)$ does not contain any grid points. We conclude from this that
\begin{enumerate}[leftmargin=*,label=(\roman*),ref=(\roman*)]
    \item\label{item:st}for infinitely many $N\in\mathbb{N}$ there must be a ball with radius $\delta$ that does not contain a grid point of $V_N$.
\end{enumerate} 

Now fix $p\in M$ and $r>0$ and suppose that $B(p,r)$ does not contain grid points of $V_N$ for infinitely many $N$. Now fix some positive non-zero continuous function $f$ which has support contained in $B(p,r)$. Then $\int f\dd \mu_N=0$ for infinitely many $N$, but $\int f\dd \overline V>0$. However, by assumption~\ref{as:empirical},
\begin{equation*}
    \int f\dd \mu_N \rightarrow \int f \dd \overline V\quad (N\rightarrow\infty).
\end{equation*}
This is a contradiction. We conclude that 
\begin{enumerate}[leftmargin=*,label=(\roman*),ref=(\roman*)]\setcounter{enumi}{1}
    \item\label{item:stst} for every fixed ball $B$ in $M$ there exists an $N_0$ such that $B$ contains grid points of $V_N$ for every $N\geq N_0$.

\end{enumerate}

To finish the argument let $B(q_1,\,\delta/2), B(q_2,\,\delta/2),\,\ldots,\,B(q_m,\,\delta/2)$ be a finite number of balls of radius $\delta/2$ that cover $M$. By~\ref{item:stst}, each of these balls will eventually contain a grid point. This means that there exists an $N_0$ such that for all $N\geq N_0$ each of these balls contains a grid point of $V_N$. Now let $N\geq N_0$ and let $p$ be any point of the manifold. Since $p$ is at distance less than $\delta/2$ from some $q_i$ and there is a grid point of $V_N$ at distance less than $\delta/2$ from $q_i$, it follows that $B(p,\delta)$ contains at least one grid point of $V_N$. This implies that every ball of radius $p$ contains at least one grid point of $V_N$, which contradicts~\ref{item:st}.
\end{proof}
\begin{proof}[Proof of Proposition~\ref{prop:marginals}]
First of all
\begin{equation*}
    \E \left\langle\sqrt N \widetilde \phi_N,f\right\rangle ^2 \stackrel{\eqref{eq:tilde_useful}}{=} \E \langle\sqrt N \phi_N,\widetilde f_N \rangle ^2
    = \frac{1}{N}(\widetilde f_N, G_N \widetilde f_N).
\end{equation*}
Recall the notations $f|_N=f|_{V_N}$, $\bar f^N={1}/{N}\sum_{i=1}^N f(p_i)$ and $f_N=f|_N-\bar f^N$. We have shown in Section~\ref{subsec:thm1} that
\begin{equation*}
    \frac{1}{N}(f|_N, G_N f|_N)\rightarrow (f,\,Gf)
\end{equation*}
(actually we have shown this for $f_N$, but since $G$ maps constant vectors to $0$ this does not make a difference). Hence by the triangular inequality it suffices to show that 
\begin{equation*}
    \left|\frac{1}{N}(\widetilde f_N, G_N \widetilde f_N)-\frac{1}{N}(f|_N, G_N f|_N)\right|\rightarrow 0.
\end{equation*}
By linearity and Cauchy--Schwarz, we see that
\begin{align}
    \left|\frac{1}{N}(f|_N, G_N f|_N)\right.&\left.-\frac{1}{N}(\widetilde f_N, G_N \widetilde f_N)\right|\leq 
    \frac{1}{N}\left( |(f|_N-\widetilde f_N,G_N f|_N)|+|(\widetilde f_N, G_N(f|_N-\widetilde f_N))|\right)\nonumber \\ 
    &\leq \frac{1}{N} \|f|_N-\widetilde f_N\| \|G_N\| \|f|_N\| + \frac{1}{N}\|\widetilde f_N\|\|G_N\| \|f|_N-\widetilde f_N\|.\label{difference}
\end{align}
Now we see that
\begin{equation*}
    \left(\frac{1}{\sqrt N} \|f|_N\|\right)^2 = \frac{1}{N}\sum_{i=1}^N f(p_i)^2 \leq \|f\|_{L^\infty}^2
\end{equation*}
and 
\begin{equation*}
    \left(\frac{1}{\sqrt N} \|\widetilde f_N\|\right)^2 = \frac{1}{N}\sum_{i=1}^N \left(\frac{1}{v_i}\int_{C_i}f(p)\overline V(\dd p)\right)^2 \leq \frac{1}{N}\sum_{i=1}^N \|f\|_{L^\infty}^2 = \|f\|_{L^\infty}^2.
\end{equation*}
Also $\|G_N\| = ({\lambda_{2}^N})^{-1}$. Further, we see that for all $p\in C_i$,
\begin{equation*}
    f(p_i) - L_f\epsilon_N \leq f(p_i)- L_f d(p,p_i)\leq f(p)\leq f(p_i)+ L_f d(p,p_i) \leq f(p_i) + L_f\epsilon_N,
\end{equation*}
which implies that 
\begin{equation*}
    \left| f(p_i)-\frac{1}{v_i}\int_{C_i}f(p)\overline V(\dd p)\right|\leq \frac{1}{v_i}\int_{C_i}|f(p_i)-f(p)|\overline V(\dd p) \leq L_f \epsilon_N.
\end{equation*}
Now we see that 
\begin{equation*}
    \left(\frac{1}{\sqrt N}\|f|_N-\widetilde f_N\|\right)^2 = \frac{1}{N}\sum_{i=1}^N \left( f(p_i)-\frac{1}{v_i}\int_{C_i}f(p)\overline V(\dd p)\right)^2\leq \frac{1}{N}\sum_{i=1}^N L_f^2 \epsilon_N^2 = L_f^2 \epsilon_N^2,
\end{equation*}
which goes to $0$ as $N\rightarrow\infty$.
Putting everything together, we deduce that~\eqref{difference} is bounded by
\begin{equation*}
    L_f\epsilon_N\frac{1}{\lambda_{2}^N} \|f\|_\infty + \|f\|_\infty \frac{1}{\lambda_{2}^N} L_f\epsilon_N=  \frac{2\epsilon_N L_f\|f\|_{L^\infty}}{\lambda_{2}^N},
\end{equation*}
which goes to $0$ as $N\rightarrow\infty$, since by lemma~\ref{epsilonlemma} $\epsilon_N\rightarrow 0$ and by Assumption~\ref{as:gap} $\inf_N\lambda_{2}^N>0$.
\end{proof}